\documentclass[12pt]{amsart}

\usepackage{amssymb,amsmath,amsfonts,amsthm}
\usepackage{graphics,enumerate,mathrsfs,mathtools,soul}
\usepackage{dsfont,bm}

\usepackage{tikz-cd}
\usepackage{circuitikz}
\usetikzlibrary{positioning, arrows,arrows.meta, scopes}
\usetikzlibrary{calc,decorations.markings}
\usetikzlibrary{shapes,snakes}

\definecolor{cadmiumgreen}{rgb}{0.0, 0.42, 0.24}
\usepackage[
colorlinks, citecolor=cadmiumgreen,
pdfauthor={Harry Richman, Farbod Shokrieh, Chenxi Wu}, 
pdfstartview ={FitV},
]{hyperref}
\hypersetup{
pdftitle={Counting two-forests and random cut size via potential theory},
}

\usepackage[
msc-links,
nobysame,
lite,
]{amsrefs} 

\usepackage[normalem]{ulem}
\usepackage[left=3.5cm,top=3.5cm,right=3.5cm]{geometry}
\mathtoolsset{showonlyrefs}

\newtheorem{thm}{Theorem}[section]
\newtheorem*{thm*}{Theorem}

\newtheorem{prop}[thm]{Proposition}
\newtheorem{lem}[thm]{Lemma}
\newtheorem{cor}[thm]{Corollary}

\theoremstyle{definition}

\newtheorem{dfn}[thm]{Definition}
\newtheorem*{notation*}{Notation}

\theoremstyle{remark}
\newtheorem{eg}[thm]{Example}
\newtheorem{rmk}[thm]{Remark}
\theoremstyle{definition}
\numberwithin{equation}{section}

\newcommand{\RR}{\mathbb{R}}
\newcommand{\ZZ}{\mathbb{Z}}

\newcommand{\EE}{\mathbb{E}}

\newcommand{\dd}{{\rm d}}

\newcommand{\Nvert}{\mathcal{N}_{V}}
\newcommand{\Nedge}{\mathcal{N}_{E}}

\newcommand{\tr}{\intercal}
\newcommand{\cut}{\partial}
\newcommand{\forests}{\mathcal{F}}

\newcommand{\boldmu}{\bm{\mu}}
\newcommand{\bone}{\mathbf{1}}
\newcommand{\avgdeg}{\mathrm{avg.deg}}
\newcommand{\secondsum}{\gamma}
\newcommand{\mucan}{\mu_{\rm can}}

\subjclass[2020]{
\href{https://mathscinet.ams.org/msc/msc2020.html?t=05C30}{05C30},
\href{https://mathscinet.ams.org/msc/msc2020.html?t=31C20}{31C20},
\href{https://mathscinet.ams.org/msc/msc2020.html?t=82B20}{82B20},
\href{https://mathscinet.ams.org/msc/msc2020.html?t=05B35}{05B35},
\href{https://mathscinet.ams.org/msc/msc2020.html?t=14T15}{14T15}}

\date{\today}

\begin{document}
\title{Counting two-forests and random cut size via potential theory}
\author{Harry Richman}
\email{\href{mailto:hrichman@uw.edu}{hrichman@uw.edu}}
\author{Farbod Shokrieh}
\email{\href{mailto:farbod@uw.edu}{farbod@uw.edu}}
\author{Chenxi Wu}
\email{\href{mailto:cwu367@math.wisc.edu}{cwu367@math.wisc.edu}}

\begin{abstract}
We prove a lower bound on the number of spanning two-forests in a graph, in terms of the number of vertices, edges, and spanning trees.
This implies an upper bound on the average cut size of a random two-forest.
The main tool is an identity relating the number of spanning trees and two-forests to pairwise effective resistances in a graph.
Along the way, we make connections to potential theoretic invariants on metric graphs.
\end{abstract}

\maketitle

\setcounter{tocdepth}{1}
\tableofcontents

\section{Introduction} 
\renewcommand*{\thethm}{\Alph{thm}}

For a connected graph $G$, 
let $\kappa(G)$ denote the number of spanning trees 
and let $\kappa_2(G)$ denote the number of  two-component spanning forests.
We use the term {\em two-forest} to mean a two-component spanning forest.
The first main result of this paper is to prove the following lower bound on $\kappa_2(G)$.
\begin{thm}[Theorem~\ref{thm:main-two-forests}]
\label{thm:main-1}
For any connected graph $G = (V,E)$,
\begin{equation}
\label{eq:main-1}
    \frac{\kappa_2(G)}{\kappa(G)} \geq \frac{(|V| - 1)^2}{4|E|} \, .
\end{equation}
\end{thm}

Given a two-forest $F$, the {\em (edge) cut set} of $F$
is defined as the set of edges which have one endpoint in each component of $F$;
we use $\cut F$ to denote the cut set.
We also obtain the following upper bound on the average size of $\cut F$. 
\begin{thm}[Theorem~\ref{thm:main-cut-set}]
\label{thm:main-2}
Suppose $G = (V,E)$ is a connected graph.
Then, for a uniformly random two-forest $F$, the expected size of the cut set $\cut F$ satisfies
\begin{equation}
\label{eq:main-2}
	\mathbb{E}( |\cut F| ) \;\leq\; 2(\avgdeg) \left(1 + \frac{1}{|V|-1}  \right) \, ,
\end{equation}
where $\avgdeg = 2|E| / |V|$ is the average degree of $G$.
\end{thm}
Here by ``uniformly random two-forest'' we mean a random two-forest under the uniform probability distribution on all possible two-forests.

If $G$ is $d$-regular on $n$ vertices, Theorem \ref{thm:main-2} states
\begin{equation}
\label{eq:main-2-regular}
	\mathbb{E}( |\cut F| ) \leq 2d \left(1 + \frac{1}{n-1}  \right) \, .
\end{equation}

The main ingredient in proving the inequalities in Theorems~\ref{thm:main-1} and \ref{thm:main-2} is a new result relating the number of spanning trees and two-forests to potential-theoretic invariants of the graph. 
Let $r(x,y)$ denote the {\em effective resistance} between $x, y \in V$ (see \S\ref{sec:resistance} for a definition).

\begin{thm}[Theorem~\ref{thm:tau-to-forests}]
\label{thm:main-3}
Let $G = (V, E)$ be a connected graph, and fix a vertex $q \in V$. 
Then
\[
	\frac{\kappa_2(G)}{\kappa(G)}  = \frac{1}{4} \sum_{e \in E} {r(e^+, e^-)^2} + \frac{3}{4} \sum_{e \in E} {\left( r(e^+,q) - r(e^-,q) \right)^2} \, ,
\]
where $e^+$ and $e^-$ denote the endpoints of an edge $e$.
\end{thm}

The bound in Theorem~\ref{thm:main-1} is obtained by applying the Cauchy--Schwarz inequality and Foster's identity to the first sum in Theorem~\ref{thm:main-3},
and observing that the second sum is nonnegative.
We remark that the second summation, which we will denote by $3\, \secondsum(G)$, is closely related to a certain ``capacity'' of $G$ called the {\em tau constant} $\tau(G)$. 
See \S\ref{sec:cap} for a brief discussion. We pursue this connection with capacity theory further in an upcoming paper (\cite{richman-shokrieh-wu}).

\subsection{Related work}
Estimates of $\mathbb E(|\cut F|)$ for a uniformly random two-forest $F$ were previously studied by Kassel, Kenyon, and Wu in \cite{kassel-kenyon-wu}.
In \cite[Equation (10)]{kassel-kenyon-wu}, it is shown that for (finite graphs limiting to) $\ZZ^r$, which has degree $d = 2r$,
the expected cut set size $|\cut F|$ 
for a random two-forest is $4r + o(1)$.
This gives a family of examples where the bound in Theorem~\ref{thm:main-2}
is sharp, up to the leading term.

Theorem~\ref{thm:main-1} can be viewed in the context of matroid theory as follows: a connected graph on $n$ vertices and $m$ edges defines a graphic matroid on $m$ elements, which has rank $r = n-1$.
The ratio $\kappa_2(G) / \kappa(G)$ is equal to $I_{r-1}/I_r$ for this matroid,
where $I_k$ denotes the number of independent sets of size $k$ in a matroid.
Moreover for a simple graph, the graphic matroid satisifes $1/m = I_0/I_1$.

There has been much work focused on 
proving bounds among the ratios $I_{k-1} / I_k$ as $k$ varies, for an arbitrary matroid.
Mason conjectured in \cite{mason} that these ratios satisfy
\begin{equation}
\label{eq:mason-bound}
\frac{I_{k}}{I_{k+1}} \geq \left(1 + \frac1{k}\right)\left(1 + \frac1{m-k}\right) \frac{I_{k-1}}{I_{k}}
\end{equation}
for each $k\geq 1$, for any matroid on $m$ elements.
In particular, these bounds imply
\begin{equation}
\label{eq:mason-combined}
\frac{I_{r-1}}{I_r} \geq \frac{r}{m-r+1}
\end{equation}
for a simple matroid of rank $r$.
Mason's conjecture \eqref{eq:mason-bound} was recently proved by Anari--Liu--Oveis-Gharan--Vinzant~\cite{anari-etal} and by Br\"{a}nd\'{e}n--Huh~\cite{branden-huh}.

Note, however, that Theorem~\ref{thm:main-1} gives a stronger bound 
\[
	\frac{I_{r-1}}{I_r} \geq \frac{r^2}{4m}
\]
for graphic matroids, when $r \geq 17$ and $m - r \geq 4$ (in graph theoretic terms, when $|V| \geq 18$ and $|E| - |V| \geq 3$). 
It may be of interest to investigate whether bounds on $I_{k-1}/I_{k}$ can be strengthened for graphic matroids beyond \eqref{eq:mason-bound}, for general values of $k$. 

\subsection*{Structure of the paper}
In \S\ref{sec:graphs}, we set our notation and terminology about graphs.
In \S\ref{sec:forests}, we describe some basic results related to the enumeration of rooted spanning forests, as well as the expected size of cut sets of two-forests.
In \S\ref{sec:potential}, we start by reviewing the notions of effective resistance and potential kernel functions and their relation with the enumeration of rooted spanning forests. 
We then outline some useful properties of the resistance matrix of graphs, 
and introduce the gamma constant $\secondsum(G)$ (which appears in the identity in Theorem~\ref{thm:main-3}) in \S\ref{sec:gamma}.
In \S\ref{sec:thmC} we prove Theorem~\ref{thm:main-3}, and in \S\ref{sec:thmAB} we prove Theorems~\ref{thm:main-1} and \ref{thm:main-2}. 
We conclude the paper with a series of examples in \S\ref{sec:examples}. 
We additionally have included a basic running example throughout the paper for the ``house graph'' depicted in Figure~\ref{fig:house}.

\subsection*{Acknowledgements}
We would like to thank Cynthia Vinzant for helpful discussions. 
The first and second named authors were partially supported by the Transatlantic Research Partnership of the Embassy of France in the United States and the FACE Foundation. The second-named author was partially supported by NSF CAREER DMS-2044564 grant. 
The third-named author was partially supported by the Simons Collaboration grant. 

\renewcommand*{\thethm}{\arabic{section}.\arabic{thm}}

\section{Graphs and matrices} \label{sec:graphs}
A graph is assumed to have a finite number of vertices and edges.
We allow parallel edges (multi-edges) but no loop edges.
We assume all graphs are connected unless otherwise specified.

We assume that all graphs in this paper are oriented.
Formally, $G$ is equipped with a pair of maps $h: E \to V$ and $t: E \to V$,
such that $h(e)$ and $t(e)$ are the endpoints of $e$.
We abbreviate $h(e) = e^+$ and $t(e) = e^-$.

The {\em genus} of a graph $G = (V,E)$ is $g(G) = |E| - |V| + 1$.
Given $v\in V$, 
let $\Nedge(v)$ denote the set of edges which are incident to $v$ (in either orientation),
and let $\Nvert(v)$ denote the multiset of vertices which are adjacent to $v$ with multiplicity given by the number of connecting edges.

We assume each graph $G = (V,E)$ is implicitly equipped with a linear ordering on its vertex set, $V \cong [n] = \{1,2,\ldots,n\}$,
which allows us to order the rows and columns of vertex-indexed graph matrices.
The choices of linear ordering on $V$ and orientation on $E$ are auxiliary in the sense that the quantities we study will not depend on which order or orientation are chosen.
The linear ordering on $V$ can also be used to induce a particular orientation on the edge set, by orienting each edge towards its larger endpoint. 

Given a graph $G = (V,E)$, let $B = B(G)$ denote the {\em (signed) incidence matrix} of $G$. 
This is the matrix in $\RR^{V\times E}$ with entries
\begin{equation}
B_{v,e} = \begin{cases}
1 &\text{if } v = e^+ , \\
-1 &\text{if } v = e^- ,\\
0 &\text{otherwise}.
\end{cases}
\end{equation}
Let $L = L(G)$ denote the Laplacian matrix of $G$ defined by
$L = B B^\tr$.
Its entries are
\[
L_{v,w} = \begin{cases}
\deg(v) &\text{if } v = w , \\
-\left| \{\text{edges between $v$ and $w$}\} \right| &\text{if } v \neq w.
\end{cases}
\]

\begin{eg}
Consider the ``house graph'', shown in Figure~\ref{fig:house}.
We will use this graph as a running example throughout the paper.
\begin{figure}[h]
\centering
	\begin{tikzpicture}[scale=0.5]
	\coordinate (1) at (0,0);
	\coordinate (2) at (2,0);
	\coordinate (3) at (0,1.8);
	\coordinate (4) at (2,1.8);
	\coordinate (5) at (1,3);
	
	\foreach \c in {1,2,3,4,5} {
		\filldraw[black] (\c) circle (2pt);
	}

	\node[left] at (1) {$1$};
	\node[right] at (2) {$2$};
	\node[right] at (4) {$4$};
	\node[left] at (3) {$3$};
	\node[above right] at (5) {$5$};
	
	\draw (1) -- (2) -- (4) -- (3) -- cycle;
	\draw (3) -- (5) -- (4);
	\end{tikzpicture}
	\caption{House graph.}
	\label{fig:house}
\end{figure}
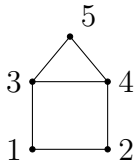

The incidence matrix and Laplaican matrix are
\medskip
\[
\tiny{
B = \begin{pmatrix}
-1 & -1 & 0 & 0 & 0 & 0 \\
1 & 0 & -1 & 0 & 0 & 0 \\
0 & 1 & 0 & -1 & -1 & 0 \\
0 & 0 & 1 & 1 & 0 & -1 \\
0 & 0 & 0 & 0 & 1 & 1 
\end{pmatrix}
,
\qquad\qquad
L = \begin{pmatrix}
2 & -1 & -1 & 0 & 0 \\
-1 & 2 & 0 & -1 & 0 \\
-1 & 0 & 3 & -1 & -1 \\
0 & -1 & -1 & 3 & -1 \\
0 & 0 & -1 & -1 & 2
\end{pmatrix} \, .
}
\]
\end{eg}

\section{Trees and forests} \label{sec:forests}

\subsection{Rooted spanning forests}
Given a graph $G$, a {\em spanning tree} is a subgraph which is connected, contains no cycles, and contains all vertices of $G$.
A {\em spanning forest} is a subgraph which contains no cycles and includes all vertices of $G$.
An {\em $r$-forest} is a spanning forest which has exactly $r$ connected components. 
We let $\forests_r(G)$ denote the set of $r$-forests 
and $\kappa_r(G)$ the number of $r$-forests.
We are primarily interested in $\kappa(G) \vcentcolon = \kappa_1(G)$ and $\kappa_2(G)$.

Given a vertex subset $S \subset V$,
an {\em $S$-rooted spanning forest} is a spanning forest which has exactly one vertex from $S$ in each connected component (so the number of components is equal to $|S|$).
If $S = \{p_1,\ldots,p_r\}$, we let $\kappa_r(p_1|p_2|\cdots|p_r)$ denote the number of $S$-rooted spanning forests.
Given $x,y,q \in V(G)$,
let $\kappa_2(xy|q)$ denote the number of two-forests which have $x$ and $y$ in one component and $q$ in the other component.
Note, in particular, that 
$\kappa_2(xy|q) = \kappa_2(yx|q)$
and $\kappa_2(xq|q) = 0$.

\begin{eg}
Suppose $G$ is the house graph (Figure~\ref{fig:house}). Then $\kappa(G)= 11$ and $\kappa_2(G) = 19$.
The two-forests of $G$ are shown in Figure~\ref{fig:house-forests}.
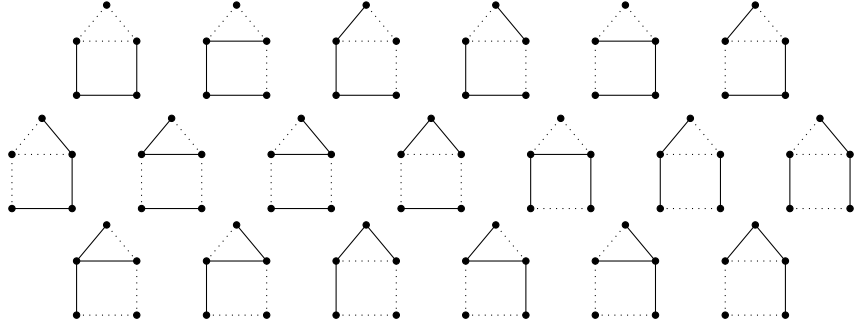
\begin{figure}[h]
\[
\begin{tikzpicture}[scale=0.4]
	\coordinate (1) at (0,0);
	\coordinate (2) at (2,0);
	\coordinate (3) at (0,1.8);
	\coordinate (4) at (2,1.8);
	\coordinate (5) at (1,3);
	\foreach \c in {1,2,3,4,5} {
		\filldraw[black] (\c) circle (3pt);
	}

	\draw (1) -- (2); 
	\draw (1) -- (3);
	\draw (2) -- (4);
	\draw[dotted] (3) -- (4);
	\draw[dotted] (3) -- (5);
	\draw[dotted] (4) -- (5);
\end{tikzpicture}
\qquad
\begin{tikzpicture}[scale=0.4]
	\coordinate (1) at (0,0);
	\coordinate (2) at (2,0);
	\coordinate (3) at (0,1.8);
	\coordinate (4) at (2,1.8);
	\coordinate (5) at (1,3);
	
	\foreach \c in {1,2,3,4,5} {
		\filldraw[black] (\c) circle (3pt);
	}

	\draw (1) -- (2); 
	\draw (1) -- (3);
	\draw[dotted] (2) -- (4);
	\draw (3) -- (4);
	\draw[dotted] (3) -- (5);
	\draw[dotted] (4) -- (5);
\end{tikzpicture}
\qquad
\begin{tikzpicture}[scale=0.4]
	\coordinate (1) at (0,0);
	\coordinate (2) at (2,0);
	\coordinate (3) at (0,1.8);
	\coordinate (4) at (2,1.8);
	\coordinate (5) at (1,3);
	
	\foreach \c in {1,2,3,4,5} {
		\filldraw[black] (\c) circle (3pt);
	}

	\draw (1) -- (2); 
	\draw (1) -- (3);
	\draw[dotted] (2) -- (4);
	\draw[dotted] (3) -- (4);
	\draw (3) -- (5);
	\draw[dotted] (4) -- (5);
\end{tikzpicture}
\qquad
\begin{tikzpicture}[scale=0.4]
	\coordinate (1) at (0,0);
	\coordinate (2) at (2,0);
	\coordinate (3) at (0,1.8);
	\coordinate (4) at (2,1.8);
	\coordinate (5) at (1,3);
	
	\foreach \c in {1,2,3,4,5} {
		\filldraw[black] (\c) circle (3pt);
	}

	\draw (1) -- (2); 
	\draw (1) -- (3);
	\draw[dotted] (2) -- (4);
	\draw[dotted] (3) -- (4);
	\draw[dotted] (3) -- (5);
	\draw (4) -- (5);
\end{tikzpicture}
\qquad
\begin{tikzpicture}[scale=0.4]
	\coordinate (1) at (0,0);
	\coordinate (2) at (2,0);
	\coordinate (3) at (0,1.8);
	\coordinate (4) at (2,1.8);
	\coordinate (5) at (1,3);
	
	\foreach \c in {1,2,3,4,5} {
		\filldraw[black] (\c) circle (3pt);
	}

	\draw (1) -- (2); 
	\draw[dotted] (1) -- (3);
	\draw (2) -- (4);
	\draw (3) -- (4);
	\draw[dotted] (3) -- (5);
	\draw[dotted] (4) -- (5);
\end{tikzpicture}
\qquad
\begin{tikzpicture}[scale=0.4]
	\coordinate (1) at (0,0);
	\coordinate (2) at (2,0);
	\coordinate (3) at (0,1.8);
	\coordinate (4) at (2,1.8);
	\coordinate (5) at (1,3);
	
	\foreach \c in {1,2,3,4,5} {
		\filldraw[black] (\c) circle (3pt);
	}

	\draw (1) -- (2); 
	\draw[dotted] (1) -- (3);
	\draw (2) -- (4);
	\draw[dotted] (3) -- (4);
	\draw (3) -- (5);
	\draw[dotted] (4) -- (5);
\end{tikzpicture}
\]
\[
\begin{tikzpicture}[scale=0.4]
	\coordinate (1) at (0,0);
	\coordinate (2) at (2,0);
	\coordinate (3) at (0,1.8);
	\coordinate (4) at (2,1.8);
	\coordinate (5) at (1,3);
	
	\foreach \c in {1,2,3,4,5} {
		\filldraw[black] (\c) circle (3pt);
	}

	\draw (1) -- (2); 
	\draw[dotted] (1) -- (3);
	\draw (2) -- (4);
	\draw[dotted] (3) -- (4);
	\draw[dotted] (3) -- (5);
	\draw (4) -- (5);
\end{tikzpicture}
\qquad
\begin{tikzpicture}[scale=0.4]
	\coordinate (1) at (0,0);
	\coordinate (2) at (2,0);
	\coordinate (3) at (0,1.8);
	\coordinate (4) at (2,1.8);
	\coordinate (5) at (1,3);
	
	\foreach \c in {1,2,3,4,5} {
		\filldraw[black] (\c) circle (3pt);
	}

	\draw (1) -- (2); 
	\draw[dotted] (1) -- (3);
	\draw[dotted] (2) -- (4);
	\draw (3) -- (4);
	\draw (3) -- (5);
	\draw[dotted] (4) -- (5);
\end{tikzpicture}
\qquad
\begin{tikzpicture}[scale=0.4]
	\coordinate (1) at (0,0);
	\coordinate (2) at (2,0);
	\coordinate (3) at (0,1.8);
	\coordinate (4) at (2,1.8);
	\coordinate (5) at (1,3);
	
	\foreach \c in {1,2,3,4,5} {
		\filldraw[black] (\c) circle (3pt);
	}

	\draw (1) -- (2); 
	\draw[dotted] (1) -- (3);
	\draw[dotted] (2) -- (4);
	\draw (3) -- (4);
	\draw[dotted] (3) -- (5);
	\draw (4) -- (5);
\end{tikzpicture}
\qquad
\begin{tikzpicture}[scale=0.4]
	\coordinate (1) at (0,0);
	\coordinate (2) at (2,0);
	\coordinate (3) at (0,1.8);
	\coordinate (4) at (2,1.8);
	\coordinate (5) at (1,3);
	
	\foreach \c in {1,2,3,4,5} {
		\filldraw[black] (\c) circle (3pt);
	}

	\draw (1) -- (2); 
	\draw[dotted] (1) -- (3);
	\draw[dotted] (2) -- (4);
	\draw[dotted] (3) -- (4);
	\draw (3) -- (5);
	\draw (4) -- (5);
\end{tikzpicture}
\qquad
\begin{tikzpicture}[scale=0.4]
	\coordinate (1) at (0,0);
	\coordinate (2) at (2,0);
	\coordinate (3) at (0,1.8);
	\coordinate (4) at (2,1.8);
	\coordinate (5) at (1,3);
	
	\foreach \c in {1,2,3,4,5} {
		\filldraw[black] (\c) circle (3pt);
	}

	\draw[dotted] (1) -- (2); 
	\draw (1) -- (3);
	\draw (2) -- (4);
	\draw (3) -- (4);
	\draw[dotted] (3) -- (5);
	\draw[dotted] (4) -- (5);
\end{tikzpicture}
\qquad
\begin{tikzpicture}[scale=0.4]
	\coordinate (1) at (0,0);
	\coordinate (2) at (2,0);
	\coordinate (3) at (0,1.8);
	\coordinate (4) at (2,1.8);
	\coordinate (5) at (1,3);
	
	\foreach \c in {1,2,3,4,5} {
		\filldraw[black] (\c) circle (3pt);
	}

	\draw[dotted] (1) -- (2); 
	\draw (1) -- (3);
	\draw (2) -- (4);
	\draw[dotted] (3) -- (4);
	\draw (3) -- (5);
	\draw[dotted] (4) -- (5);
\end{tikzpicture}
\qquad
\begin{tikzpicture}[scale=0.4]
	\coordinate (1) at (0,0);
	\coordinate (2) at (2,0);
	\coordinate (3) at (0,1.8);
	\coordinate (4) at (2,1.8);
	\coordinate (5) at (1,3);
	
	\foreach \c in {1,2,3,4,5} {
		\filldraw[black] (\c) circle (3pt);
	}

	\draw[dotted] (1) -- (2); 
	\draw (1) -- (3);
	\draw (2) -- (4);
	\draw[dotted] (3) -- (4);
	\draw[dotted] (3) -- (5);
	\draw (4) -- (5);
\end{tikzpicture}
\]
\[
\begin{tikzpicture}[scale=0.4]
	\coordinate (1) at (0,0);
	\coordinate (2) at (2,0);
	\coordinate (3) at (0,1.8);
	\coordinate (4) at (2,1.8);
	\coordinate (5) at (1,3);
	
	\foreach \c in {1,2,3,4,5} {
		\filldraw[black] (\c) circle (3pt);
	}

	\draw[dotted] (1) -- (2); 
	\draw (1) -- (3);
	\draw[dotted] (2) -- (4);
	\draw (3) -- (4);
	\draw (3) -- (5);
	\draw[dotted] (4) -- (5);
\end{tikzpicture}
\qquad
\begin{tikzpicture}[scale=0.4]
	\coordinate (1) at (0,0);
	\coordinate (2) at (2,0);
	\coordinate (3) at (0,1.8);
	\coordinate (4) at (2,1.8);
	\coordinate (5) at (1,3);
	
	\foreach \c in {1,2,3,4,5} {
		\filldraw[black] (\c) circle (3pt);
	}

	\draw[dotted] (1) -- (2); 
	\draw (1) -- (3);
	\draw[dotted] (2) -- (4);
	\draw (3) -- (4);
	\draw[dotted] (3) -- (5);
	\draw (4) -- (5);
\end{tikzpicture}
\qquad
\begin{tikzpicture}[scale=0.4]
	\coordinate (1) at (0,0);
	\coordinate (2) at (2,0);
	\coordinate (3) at (0,1.8);
	\coordinate (4) at (2,1.8);
	\coordinate (5) at (1,3);
	
	\foreach \c in {1,2,3,4,5} {
		\filldraw[black] (\c) circle (3pt);
	}

	\draw[dotted] (1) -- (2); 
	\draw (1) -- (3);
	\draw[dotted] (2) -- (4);
	\draw[dotted] (3) -- (4);
	\draw (3) -- (5);
	\draw (4) -- (5);
\end{tikzpicture}
\qquad
\begin{tikzpicture}[scale=0.4]
	\coordinate (1) at (0,0);
	\coordinate (2) at (2,0);
	\coordinate (3) at (0,1.8);
	\coordinate (4) at (2,1.8);
	\coordinate (5) at (1,3);
	
	\foreach \c in {1,2,3,4,5} {
		\filldraw[black] (\c) circle (3pt);
	}

	\draw[dotted] (1) -- (2); 
	\draw[dotted] (1) -- (3);
	\draw (2) -- (4);
	\draw (3) -- (4);
	\draw (3) -- (5);
	\draw[dotted] (4) -- (5);
\end{tikzpicture}
\qquad
\begin{tikzpicture}[scale=0.4]
	\coordinate (1) at (0,0);
	\coordinate (2) at (2,0);
	\coordinate (3) at (0,1.8);
	\coordinate (4) at (2,1.8);
	\coordinate (5) at (1,3);
	
	\foreach \c in {1,2,3,4,5} {
		\filldraw[black] (\c) circle (3pt);
	}

	\draw[dotted] (1) -- (2); 
	\draw[dotted] (1) -- (3);
	\draw (2) -- (4);
	\draw (3) -- (4);
	\draw[dotted] (3) -- (5);
	\draw (4) -- (5);
\end{tikzpicture}
\qquad
\begin{tikzpicture}[scale=0.4]
	\coordinate (1) at (0,0);
	\coordinate (2) at (2,0);
	\coordinate (3) at (0,1.8);
	\coordinate (4) at (2,1.8);
	\coordinate (5) at (1,3);
	
	\foreach \c in {1,2,3,4,5} {
		\filldraw[black] (\c) circle (3pt);
	}

	\draw[dotted] (1) -- (2); 
	\draw[dotted] (1) -- (3);
	\draw (2) -- (4);
	\draw[dotted] (3) -- (4);
	\draw (3) -- (5);
	\draw (4) -- (5);
\end{tikzpicture}
\]
\caption{Two-forests of the house graph.}
\label{fig:house-forests}
\end{figure}
\end{eg}

\begin{notation*}
Given a square matrix $M$,
let $M[\overline{i},\overline{j}]$ denote the matrix obtained from $M$ by deleting row $i$ and column $j$.
If $i = j$, we abbreviate 
$M[\overline{i}, \overline{i}]$
with 
$M[\overline{i}]$.
More generally, we let $M[\overline{I}, \overline{J}]$ (respectively $M[I,J]$) denote the matrix obtained from $M$ 
by deleting (respectively keeping) the $I$-indexed rows and the $J$-indexed columns.
We use $M[\overline{I}]$ as shorthand for $M[\overline{I}, \overline{I}]$,
and if $I = \{i,k\}$, we use 
$M[\overline{ik}]$
as shorthand for $M[\overline{I}]$.
\end{notation*}

As is well known, one can compute the number of spanning trees and, more generally, rooted spanning forests using the Laplacian matrix $L$.

\begin{thm}
\label{thm:matrix-tree}
Let $L$ denote the Laplacian matrix of a graph $G = (V,E)$.
\begin{enumerate}[(a)]
\item 
For any $q \in V$,
\begin{equation*}
\kappa(G) 
= \det L[\overline{q}] \, .
\end{equation*}

\item 
For any nonempty set of vertices $S = \{x_1, \ldots, x_r\} \subset V$,
\begin{equation*}
\kappa_r(x_1|x_2|\cdots|x_r) 
= \det L[\overline{S}] \, .
\end{equation*}

\item 
For any $x,y,q \in V$,
\begin{equation}
\kappa_2(xy|q) = \left| \det L[\overline{xq}, \overline{yq}] \right| \, .
\end{equation}
\end{enumerate}
\end{thm}
\begin{proof}
See Tutte~\cite[Section VI.6]{tutte}. Note that part (a) is Kirchhoff's celebrated matrix tree theorem~\cite{Kirchhoff1847}. 
\end{proof}

\begin{prop}
\label{prop:3-forest}
Let $G$ be a connected graph. Let $x,y,q\in V(G)$. 
Then 
\[
	\kappa(G) \kappa_3(x|y|q) = \kappa_2(x|q) \kappa_2(y|q) - \kappa_2(xy|q)^2 \, .
\]
\end{prop}
\begin{proof}
We may assume $x,y,q$ are pairwise distinct (otherwise both sides of the identity vanish). Let $M = L[\overline q]$. 
By Theorem~\ref{thm:matrix-tree} we have
\[
	\kappa(G) = \det M \quad , \quad
	\kappa_2(x|q) = \det M[\overline x] \quad , \quad
	\kappa_2(y|q) = \det M[\overline y] \, , 
\]
\[
	\kappa_3(x|y|q) = \det M[\overline{xy}] \quad ,
	\quad
	\pm \kappa_2(xy|q) = \det M[\overline{x},\overline{y}]
	 = \det M[\overline y, \overline{x}] \, .
\]

The identity now follows from the {\em Desnanot--Jacobi identity} (a special case of {\em Sylvester's determinant identity} -- see \cite{sylvester1851xxxvii}): for any square matrix $M$ and any indices $i\neq j$,
\[
	(\det M) (\det M[\overline{ij}]) = (\det M[\overline{i}]) (\det M[\overline{j}]) - (\det M[\overline{i},\overline{j}]) (\det M[\overline{j},\overline{i}]) \, .
\]
\end{proof}

The following is due to Liu and Chow \cite{liu-chow} (see also Myrvold \cite{myrvold}). 
We include a short proof here for completeness.
\begin{prop}[see \cite{liu-chow}]
\label{prop:forest-count}
Suppose $G = (V,E)$ is a connected graph and fix a vertex $q \in V$. Then
\begin{equation}
	\kappa_2(G) = \sum_{v \in V} \kappa_2(v|q) - \sum_{e \in E} \kappa_3(e^+| e^-| q) \, .
\end{equation}
\end{prop}
\begin{proof}
Suppose $F$ is a two-forest whose connected components induce a vertex partition $V = A \sqcup B$, such that $q \in B$.
Then 
\begin{itemize} 
\item 
the first sum 
$\displaystyle \sum_{v \in V} \kappa_2(v|q)$
counts $F$ with multiplicity $|A|$, and
\item 
the second sum
$\displaystyle \sum_{e \in E} \kappa_3(e^+| e^-| q)$
counts $F \setminus e$ once for each edge $e$ in the $A$-component of $F$,
hence counts $F$ with multiplicity $(|A|-1)$.
\end{itemize}
Therefore, the right-hand side counts each two-forest $F$ with multiplicity one, yielding $\kappa_2(G)$.
\end{proof}

\subsection{Cut sets of two-forests}

Given a two-forest $F$, suppose the components of $F$ induce a partition $V = A \sqcup B$ on the vertex set .
The {\em cut set} $\cut F$ of this two-forest is the subset of edges which connect $A$ and $B$, i.e. 
\begin{equation}
	\cut F = \{ e \in E : e = \{u , v\} \text{ for } u \in A \text{ and } v \in B \}\, .
\end{equation}
The {\em cut size} of $F$ refers to the cardinality $|\cut F|$.
If $F$ is a {\em uniformly random two-forest}, then the expected value of $|\cut F|$ is  
\begin{equation}
\label{eq:random-forest-cut}
	\mathbb E( |\cut F|) = \frac{\sum_{F \in \mathcal F_2(G)} |\cut F|} {\kappa_2(G)} \, .
\end{equation}

\begin{eg}
Suppose $G$ is the house graph (Figure~\ref{fig:house}).
A two-forest $F$ in $G$ can have cut size $|\cut F|=2$ or $3$.
An example of each case is shown in Figure~\ref{fig:house-cut}.
The edges in $F$ are solid, and the edges in $\cut F$ are marked with an empty circle.
\begin{figure}[h]
\begin{minipage}{0.45\textwidth}
\centering
\begin{tikzpicture}[scale=0.4]
	\coordinate (1) at (0,0);
	\coordinate (2) at (2,0);
	\coordinate (3) at (0,1.8);
	\coordinate (4) at (2,1.8);
	\coordinate (5) at (1,3);
	
	\foreach \c in {1,2,3,4,5} {
		\filldraw[black] (\c) circle (3pt);
	}

	\draw (1) -- (2); 
	\draw[dotted] (1) -- node{$\circ$} (3);
	\draw[dotted] (2) -- node{$\circ$} (4);
	\draw[dotted] (3) -- (4);
	\draw (3) -- (5);
	\draw (4) -- (5);
\end{tikzpicture}

\end{minipage}
\begin{minipage}{0.45\textwidth}
\centering
\begin{tikzpicture}[scale=0.4]
	\coordinate (1) at (0,0);
	\coordinate (2) at (2,0);
	\coordinate (3) at (0,1.8);
	\coordinate (4) at (2,1.8);
	\coordinate (5) at (1,3);
	
	\foreach \c in {1,2,3,4,5} {
		\filldraw[black] (\c) circle (3pt);
	}

	\draw (1) -- (2); 
	\draw[dotted] (1) -- node{$\circ$} (3);
	\draw (2) -- (4);
	\draw[dotted] (3) -- node{$\circ$} (4);
	\draw (3) -- (5);
	\draw[dotted] (4) -- node{$\circ$} (5);
\end{tikzpicture}

\end{minipage}
\caption{Two-forests of cut size two (left) and three (right).}
\label{fig:house-cut}
\end{figure}
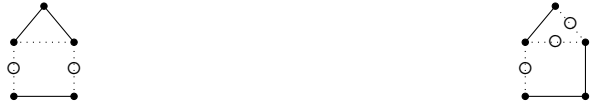

\noindent From Figure~\ref{fig:house-forests}, one can see that there are $13$ two-forests with $|\cut F| = 2$ and $6$ two-forests with $|\cut F| = 3$. 
The expected cut size is
\[
  \EE(|\cut F|) = \frac{{13 \times 2 + 6 \times 3}}{19} = \frac{44}{19}.
\]
\end{eg}

The following result of Kassel, Kenyon, and Wu relates the counts of trees and two-forests to the average cut size of a two-forest. 
We include a short proof here for completeness.
\begin{lem}[see \cite{kassel-kenyon-wu}]
\label{lem:average-cut}
Let $G=(V, E)$ be a connected graph, and let $F$ be a uniformly random two-forest in $G$.
Then
\[
	\kappa_2(G) \mathbb E(|\cut F|) 
	= \kappa(G)\left( |V| - 1\right) \, .
\]
\end{lem}
\begin{proof}
This identity reflects the natural bijection between the two sets
\[
	\{(F, f) \colon F \in \forests_2(G),\, f \in \cut F \} \quad \text{and} \quad 
	\{(T, e) \colon T \in \forests_1(G),\, e \in T \}\, ,
\]
given by adding $f$ to $F$ and by removing $e$ from $T$.
\end{proof}

\section{Potential theory on graphs} 
\label{sec:potential}

\subsection{Resistance and potential kernel}
\label{sec:resistance}

Given a graph $G = (V,E)$, we may consider $G$ as an electrical network where each edge is a wire containing a resistor of unit resistance.
Let $r(x,y)$ denote the {\em effective resistance} between $x,y \in V$.
This resistance can be expressed in terms of the number of spanning trees and rooted spanning forests:
\begin{equation}
\label{eq:resistance-trees}
	r(x,y) = \frac{\kappa_2(x|y)}{\kappa(G)}
	\qquad\text{for any }x,y\in V.
\end{equation}
This well-known identity is due to Kirchhoff~\cite{Kirchhoff1847}; for a more modern treatment see, e.g., Biggs~\cite[Section 17]{biggs}.

More generally, let $j_q: V \times V \to \RR$ denote the unit {\em potential function}, or {\em potential kernel}, with respect to $q \in V$. This $j$-function is defined by: 
\begin{itemize}
\item 
$j_q(q, \cdot) = j_q(\cdot, q) = 0$,
and 
\item 
for $x \neq q$ and $y \neq q$ :
\begin{equation}
\label{eq:j-function-def}
	j_q(x,y) \;\text{ is the $(x,y)$-entry of the matrix inverse of } L[\overline{q}].
\end{equation}
\end{itemize}
Since $L[\overline{q}]$ is symmetric, we have $j_q(x, y) = j_q(y, x)$.

The values of the $j$-function can also be expressed in terms of the number of spanning trees and rooted spanning forests (see, e.g., \cite[Proposition 13.1]{biggs}):
\begin{equation}
\label{eq:potential-trees}
j_q(x,y)
= \frac{\kappa_2(xy | q)}{\kappa(G)}
\qquad\text{for any } x,y,q\in V.
\end{equation}
It follows from \eqref{eq:resistance-trees} and \eqref{eq:potential-trees} that $r(x,q) = j_q(x,x)$.

\begin{rmk}
The physical meaning of the $j$-function is as follows: 
the value $j_q(x, y)$ gives the electric potential at $x$ if one unit of current enters the network at $y$ and exits at $q$, with $q$ ``grounded'' (i.e., has zero potential).

\ctikzset{bipoles/length=.7cm}
\begin{figure}[h!]
\begin{tikzpicture}
\node [cloud, draw,cloud puffs=10.2,cloud puff arc=120, aspect=2, inner ysep=1em, gray] (cloud) at (0, 0) {};
	\fill[black] (0,-.5) circle (.1);
	\fill[black] (.7,.3) circle (.1);
	\fill[black] (-.8,0) circle (.1);
	\draw (0,-.5) to (0,-1) node[ground]{};
	\draw (0,-.9) to (1.8,-.9) to[dcisource] (1.8, .3) to (.7,.3);
	\draw[{Circle[open]}-] (-1.8,-.9) to (0,-.9);
	\draw[{Circle[open]}-] (-1.8,0) to (-.8,0);
	\fill[black] (0,-.9) circle (.05);
	\node at (-2, -.45) {\tiny{$j_q(x,y)$}};
	\node at (-2, -.9) {$-$};
	\node at (-2, 0) {$+$};
	\node at (0, -.2) {\tiny{$q$}};
	\node at (.7, 0) {\tiny{$y$}};
	\node at (-.8, .3) {\tiny{$x$}};
	\node at (2.2, -.3) {\tiny{$1$}};
\end{tikzpicture}
\caption{Electrical network interpretation of the $j$-function.}\label{fig:jfunction}
\end{figure}
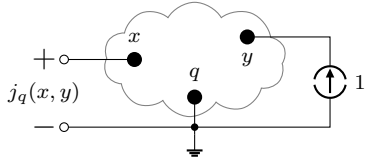

\end{rmk}

We recall some basic relations between the effective resistance and $j$-functions, which will be useful later.
\begin{prop}
\label{prop:res-identities}
Suppose $x,y,q$ are vertices of a connected graph $G$. 
\hfill
\begin{enumerate}[(a)]
\item 
$r(x,y) = j_x(y,q) + j_y(x,q)$.

\item 
$r(x,q) + r(y,q) = r(x,y) + 2 j_q(x,y)$.

\item 
$r(x,q) - r(y,q) = 2 j_x(y, q) - r(x, y)$.

\item 
$\displaystyle
	\sum_{y \in \Nvert(x)} j_x(q, y) = \begin{cases}
	1 &\text{if } x \neq q, \\
	0 &\text{if } x = q.
	\end{cases}
$

\item 
$\displaystyle (r(x,q) - r(y,q))^2= r(x,y)^2 + 4r(x,y) j_q(x,y) + 4j_q(x,y)^2 - 4r(x,q) r(y,q)$.

\end{enumerate}
\end{prop}
\begin{proof}
We make use of identities \eqref{eq:resistance-trees} and \eqref{eq:potential-trees}.
\begin{enumerate}[(a)]
\item 
For any $x,y,q \in V$, we have
\begin{equation}
\kappa_2(x|y) = \kappa_2(xq|y) + \kappa_2(x|yq)
\end{equation}
since a two-forest counted by $\kappa_2(x|y)$ must have $q$ in either the $x$-component or the $y$-component.
Then divide both sides by $\kappa(G)$.

\item 
For any $x,y,q \in V$, we have
\begin{align}
\kappa_2(x|q) + \kappa_2(y|q)
&= \kappa_2(x|yq) + \kappa_2(xy|q)
+ \kappa_2(y|xq) + \kappa_2(xy|q) \\
&= (\kappa_2(x|yq) + \kappa_2(y|xq)) + 2\kappa_2(xy|q) \\
&= \kappa_2(x|y) + 2\kappa_2(xy|q) \, .
\end{align}
Then divide both sides by $\kappa(G)$.

\item 
From (b), switch variables $x$ and $q$, and rearrange.

\item 
Since $j_x(q, y) = r(x, q) - j_q(x, y)$ by part (a), and $r(x, q) = j_q(x, x)$, 
it suffices to show that
\begin{equation}
	\sum_{y \in \Nvert(x)} (j_q(x, x) - j_q(x, y)) 
		= \begin{cases}
	1 &\text{if } x \neq q, \\
	0 &\text{if } x = q.
	\end{cases}
\end{equation}
This identity is in fact the defining property \eqref{eq:j-function-def} of the $j$-function, since the left-hand sum is
$\displaystyle 
\sum_{y \in V \setminus \{q\}} L_{x, y} \, j_q(x, y)
$
in terms of the Laplacian $L$.

Alternatively, part (d) follows from \eqref{eq:potential-trees} and the combinatorial observation that if $x \neq q$,
\[
	\sum_{y \in \Nvert(x)} \kappa_2(x | yq) = \kappa(G)
\]
via the bijection that adds the edge $\{x, y\}$ to a forest counted by $\kappa_2(x | yq)$ when $y$ is incident to $x$
(recall $\kappa_2(q | yq) = 0$ when $x=q$).

\item 
Rather than use part (c), we use part (b) to arrive at an expression that is evidently symmetric in $x$ and $y$.
We have
\begin{align}
	(r(x,q) - r(y,q))^2 &= (r(x,q) + r(y,q))^2 - 4 r(x,q) r(y,q) \\
	&= (r(x,y) + 2 j_q(x,y))^2 - 4 r(x,q) r(y,q) \, ,
\end{align}
where the second equality uses part (b).
Then expand to the desired identity.
\qedhere
\end{enumerate}
\end{proof}

\begin{rmk}In the language used for hyperbolic metric spaces,
part (b) of the above proposition says that $j_q(x,y)$ is equal to the {\em Gromov product} of $x$ and $y$ at $q$, with respect to the effective resistance metric.
Similar relations and their consequences are explored further in~\cite{dejong-shokrieh}.
\end{rmk}

The following identity on effective resistances is due to Foster~\cite{foster}.

\begin{thm}[Foster's Identity]
\label{thm:foster}
For a connected graph $G = (V,E)$,
$$\sum_{e\in E}  r(e^+, e^-) = |V| - 1 \, .$$
\end{thm}
Note that for a tree, $r(e^+,e^-) = 1$ for each edge. 
Thus Foster's identity generalizes the fact that a tree contains $|V|-1$ edges. 
Foster's identity is a ``trace formula'', as the left-hand side is the trace of the projection matrix from $\mathbb{R}^E$ onto the ``cut space'' of $G$, and the cut space has dimension $|V| - 1$ (see, e.g., \cite[Section 7]{dejong-shokrieh}).

\begin{cor} 
\label{cor:res-squared}
For a connected graph $G = (V,E)$,
\[
	\sum_{e\in E} r(e^+, e^-)^2 \;\geq\; \frac{(|V|-1)^2}{|E|} \,  .
\]
\end{cor}
\begin{proof}
This follows from Theorem~\ref{thm:foster} by the Cauchy--Schwarz inequality.
\end{proof}

\subsection{Curvature-resistance identities}
Let $R = R(G)$ denote the {\em (effective) resistance matrix} of $G$, which is the matrix in $\RR^{V\times V}$ with entries $R_{v,w} = r(v,w)$.

\begin{dfn}\label{def:mu}
Given a connected graph $G = (V,E)$, let $\boldmu = \boldmu(G)$ be 
the vector in $\RR^V$ defined by
\begin{equation}
\label{eq:mu}
\boldmu_x = 1 - \frac12 \sum_{e \in \Nedge(x)} r(e^+,e^-) 
\end{equation}
for each $x \in V$, where the sum is over all edges incident to $x$.
We call $\boldmu$ the {\em curvature vector} of $G$, following Devriendt and Lambiotte~\cite{devriendt-lambiotte}.
\end{dfn}

\begin{rmk} 
\label{rmk:curvature}
\begin{enumerate}[(i)]
\item
Although the terminology ``curvature vector'' is introduced 
in \cite{devriendt-lambiotte},
this vector has made appearances earlier in the literature on graphs and effective resistance, e.g., \cite{bapat}.
\item
The vector $\boldmu$ is a discrete analogue of the {\em canonical equilibrium measure} $\mucan$ on a metric graph; 
see Remark~\ref{rmk:analog} (ii).
\end{enumerate}
\end{rmk}

\begin{notation*}
Let $\bone$ denote the vector with all entries equal to one.
\end{notation*}

\begin{prop}
\label{prop:mu-sum}
Let $G = (V, E)$ be a connected graph with curvature vector $\boldmu$. Then $ \boldmu^\tr \bone = 1$.
\end{prop}
\begin{proof}
This is a consequence of Foster's identity (Theorem~\ref{thm:foster}):
\[
	\sum_{x \in V} \boldmu_x 
	= \sum_{x \in V} 1 - \frac12 \sum_{x \in V} \sum_{e \in \Nedge(x)} r(e^+,e^-) = |V| - \sum_{e \in E}  r(e^+,e^-) = 1. \qedhere
\]
\end{proof}

\begin{eg}
Suppose $G$ is the house graph (Figure~\ref{fig:house}). 
Then one can compute, using e.g. \eqref{eq:resistance-trees}, that
\medskip
\[
\tiny{	
R = \frac1{11} \begin{pmatrix}
	0 & 8 & 8 & 10 & 13 \\
	8 & 0 & 10 & 8 & 13 \\
	8 & 10 & 0 & 6 & 7 \\
	10 & 8 & 6 & 0 & 7 \\
	13 & 13 & 7 & 7 & 0
	\end{pmatrix}
	, \qquad\qquad
	\boldmu = \frac1{22}\begin{pmatrix}
	6 \\ 6 \\ 1 \\ 1 \\ 8
	\end{pmatrix} \,  .
}
\]
One can verify the identities
\[
	\sum_{e\in E} r(e^+,e^-) = {\textstyle\frac{8 + 8 + 8 + 6 + 7 + 7}{11}} = 4 = |V| -1 \, ,
	\qquad\qquad
	\sum_{x \in V} \boldmu_x = {\textstyle\frac{6+6+1+1+8}{22}} = 1\, .
\]
\end{eg}

We now state identities relating the resistance matrix and curvature vector.
\begin{prop}
\label{prop:curvature-resistance}
Let $G$ be a connected graph with Laplacian matrix $L$, resistance matrix $R$, and curvature vector $\boldmu$. 
Then
\begin{enumerate}[(a)]
\item 
$\displaystyle
	I + \frac{1}{2} LR = \boldmu \bone^\tr,
$

\item 
$\displaystyle
	R\boldmu = (\boldmu^\tr R \boldmu) \bone ,
$

\item 
$\displaystyle
	R + \frac{1}{2} RLR = (\boldmu^\tr R \boldmu) \bone \bone^\tr.
$

\end{enumerate}
\end{prop}
\begin{proof}
(a)
It suffices to show that the $(x,q)$-entry of the matrix $LR$ satisfies
\begin{equation}
\label{eq:LR-product}
	(LR)_{x, q} = \begin{cases}
	2 \boldmu_x &\text{if } x \neq q, \\
	-2 + 2\boldmu_x &\text{if } x = q.
	\end{cases}
\end{equation}
Computing this matrix entry, we obtain
\begin{align}
(LR)_{x, q} 
= \sum_{y \in V} L_{x, y}\, R_{y, q} 
&= \sum_{y \in \Nvert(x)} (r(x, q) - r(y, q)) \\
&= \sum_{y \in \Nvert(x)} (2 j_x(y, q) - r(x, y)) \\
&= 2 \sum_{y \in \Nvert(x)} j_x(y, q) - \sum_{y \in \Nvert(x)} r(x, y) \\
&= 2 \sum_{y \in \Nvert(x)} j_x(y, q) + (2 \boldmu_x - 2)\, ,
\end{align}
where in the second line we use Proposition~\ref{prop:res-identities} (c) and in the last line we use the definition of $\boldmu$ \eqref{def:mu}.

By Proposition~\ref{prop:res-identities} (d), the $j$-function satisfies the identity
\[
	\sum_{y \in \Nvert(x)} j_x(y, q) = \begin{cases}
	1 &\text{if } x \neq q ,\\
	0 &\text{if } x = q,
	\end{cases}
\]
so our claim~\eqref{eq:LR-product} follows.

(b) 
From the identity in part (a),
multiply both sides by $R$ on the left to obtain
\begin{equation} \label{eq:RLR}
	R + \frac{1}{2} RLR = (R \boldmu) \bone^\tr \, .
\end{equation}
The left-hand side of \eqref{eq:RLR} is symmetric, since both $R$ and $L$ are symmetric matrices.
Thus the right-hand side must also be symmetric: 
$(R \boldmu) \bone^\tr = \bone (R \boldmu)^\tr$. 
This implies that 
\begin{equation} \label{eq:Rmu}
R \boldmu = \lambda \bone \qquad\text{for some real scalar }\lambda.
\end{equation} 
It remains to show that $\lambda = \boldmu^\tr R \boldmu$. 
For this, we may use Proposition~\ref{prop:mu-sum}: 
\begin{equation} \label{eq:lambd}
	\lambda = \lambda \times 1 = \lambda (\boldmu^\tr \bone) = \boldmu^\tr (\lambda \bone) = \boldmu^\tr R \boldmu \, .
\end{equation}

(c) 
The identity follows from combining \eqref{eq:RLR}, \eqref{eq:Rmu}, and \eqref{eq:lambd}.
\end{proof}

\begin{rmk}
\hfill
\begin{enumerate}[(i)]
\item
From Proposition~\ref{prop:curvature-resistance} (a) we can deduce the identity $ L + \frac12 LRL = 0$,
which shows that $-\frac12 R$ is a {\em generalized inverse} for the Laplacian matrix $L$.
\item
The statements in Proposition~\ref{prop:curvature-resistance} can also be deduced from the work of Bapat~\cite[Theorem 3]{bapat},
where the following formula for the inverse of the resistance matrix is found:
\begin{equation}
\label{eq:res-inverse}
	R^{-1} = - \frac12 L  + \frac1{\boldmu^\tr R \boldmu} \boldmu \boldmu^\tr \, .
\end{equation}
Note that Bapat uses another vector in place of $\boldmu$ with a different normalization factor, namely $\bm{\tau} = 2 \boldmu$.
The expression \eqref{eq:res-inverse} is invariant under scaling $\boldmu$ by a nonzero factor.
\end{enumerate}
\end{rmk}

\section{Gamma constant}
\label{sec:gamma}
For a graph $G$ and a base vertex $q$, we are interested in the $\ell^2$-squared norm of the variation of the resistances $r(\cdot, q)$.
\begin{dfn}
Let $G=(V, E)$ be a connected graph, and fix a vertex $q \in V$.  The {\em gamma constant} of $G$, denoted by $\secondsum(G)$, is defined as
\begin{equation}
\label{eq:gamma-def}
	\secondsum(G) =  \frac1{4} \sum_{e \in E} (r(e^+,q) - r(e^-,q))^2 \, .
\end{equation}
\end{dfn}

\begin{rmk} 
\label{rmk:gamma}
\begin{enumerate}[(i)]
\item It is clear that $\secondsum(G) \geq 0$.
\item It is a consequence of Proposition~\ref{prop:tau-equiv} that $\secondsum(G)$ is indeed independent of the choice of the base vertex $q$.
\item The gamma constant is a discrete analogue of the {\em tau constant} on a metric graph; 
see Remark~\ref{rmk:analog} (i).
\end{enumerate}
\end{rmk}

\begin{prop}
\label{prop:tau-equiv}
Given a connected graph $G$,
let $L$ denote the Laplacian matrix,
$B$ the incidence matrix,
$R$ the resistance matrix,
and $\boldmu$ the curvature vector \eqref{eq:mu}.
\hfill
\begin{enumerate}[(a)]
\item 
All diagonal entries of $ \frac14 RLR$ are equal to $\secondsum(G)$.
\item 
$\secondsum(G) = \frac12 \boldmu^\tr R \boldmu$.
\item 
All vector entries of $\frac12 R \boldmu$ are equal to $\secondsum(G)$.
\end{enumerate}
\end{prop}
\begin{proof}
\begin{enumerate}[(a)]
\item 
It is clear from the defining equation \eqref{eq:gamma-def} that $\secondsum$ is equal to the $q$-th diagonal entry of $\frac14 (RB) (RB)^\tr = \frac14 (RB)(B^\tr R)$.
Then apply the matrix identity $L = BB^\tr$.

\item 
Proposition~\ref{prop:curvature-resistance} (c) implies that 
\[
	\frac14 RLR = - \frac12 R + \frac12(\boldmu^\tr R \boldmu) \bone \bone^\tr \, .
\] 
Since diagonal entries of $R$ are all zero, each diagonal entry of $\frac14 RLR$ is equal to $\frac12 \boldmu^\tr R \boldmu$. The result follows from part (a).

\item 
Proposition~\ref{prop:curvature-resistance} (b) implies that
$\displaystyle
\frac12 R\boldmu = \frac12 (\boldmu^\tr R \boldmu) \bone
$.
The result follows from part (b).
\end{enumerate}
\end{proof}

\begin{rmk}
\begin{enumerate}[(i)]
\item
Part (a) of Proposition~\ref{prop:tau-equiv} suggests the following alternate summation formula for the gamma constant:
\[
	\secondsum(G) = \frac{1}{4}\sum_{x,y \in V} L_{x,y}\, r(x,q) r(y,q) \, .
\]

\item 
The gamma constant $\secondsum(G)$ is independently studied by Devriendt~\cite[Section 3.4 and Corollary 3.10]{devriendt} under the name {\em resistance radius}, denoted by~$\sigma^2$.
\end{enumerate}
\end{rmk}

\subsection{Connection to capacity theory of metric graphs} \label{sec:cap}
Here we briefly describe the connection between $\secondsum(G)$ and the tau constant $\tau(\Gamma)$ of metric graphs. 
Historically, motivation for studying $\tau(\Gamma)$ came from arithmetic geometry (see, e.g., \cite{CR}, \cite{cinkir-inventiones}, \cite{dejong-delta}).

Let $\Gamma$ be the {\em metric graph} obtained from $G$ by replacing edges of $G$ with line segments of length $1$. 
The effective resistance $r(x, y)$ extends to a continuous function as $x, y$ vary over the entire metric graph $\Gamma$.
In terms of this resistance function, one has the following expression for the tau constant (see, e.g., \cite[Lemma 14.4]{BR-2007}):
fix a point $q \in \Gamma$, then
\begin{equation}
\label{eq:tau-1}
	\tau(\Gamma) = \frac{1}{4}\int_{\Gamma}{\left( \frac{\partial}{\partial x} r(x,q)\right)^2 \dd x} \, ,
\end{equation}
where $\dd x$ denotes the (piecewise) Lebesgue measure on $\Gamma$.

Alternatively, $\tau(\Gamma)$ can be interpreted as a certain ``capacity'', with respect to the potential kernel  $\frac12 r(x,y)$. 
It is shown in \cite[Theorem 2.11]{CR} that there exists a unique measure $\mucan$ on $\Gamma$ having total mass $1$, such that 
\begin{equation}
\label{eq:tau-2}
	\frac12 \int_{\Gamma} r(x, y) \, \dd\mucan(x)
	\qquad\text{is constant as $y$ varies.}
\end{equation}
This measure $\mucan$ is known as the {\em canonical measure} of $\Gamma$. 
In terms of $\mucan$, 
\begin{equation}
\label{eq:tau-3}
	\tau(\Gamma) = \frac12 \int \!\int_{\Gamma \times \Gamma} r(x, y) \, \dd\mucan(x)\, \dd\mucan(y)
\end{equation}
(\cite[Corollary 14.2]{BR-2007}).

\begin{rmk} \label{rmk:analog}
\begin{enumerate}[(i)]
\item
Comparing the expression for the tau constant in \eqref{eq:tau-1} with the expression for the gamma constant in \eqref{eq:gamma-def}, one can see that $\secondsum(G)$ is a natural discrete analogue of $\tau(\Gamma)$.
\item
Comparing the expression for the tau constant in \eqref{eq:tau-2} (resp. \eqref{eq:tau-3}) with the expression for the gamma constant in Proposition~\ref{prop:tau-equiv} (c) (resp. Proposition~\ref{prop:tau-equiv} (b)), one can see that $\boldmu$ is a natural discrete analogue of $\mucan$.
\end{enumerate}
\end{rmk}

Let $\Gamma$ be the metric graph obtained from $G$ by replacing edges of $G$ with unit length line segments as before
and let $\tau(G) = \tau(\Gamma)$.
One can show that ${\secondsum(G) \leq \tau(G)}$ for any graph. 
In fact, their difference has a concrete expression as a sum of squares. 
Suppose we define
\[
	\eta(G) = \frac{1}{12} \sum_{e \in E} \left(1 - r(e^+, e^-)\right)^2 \, .
\]
Then the following formula for $\tau(G)$ holds in terms of the combinatorics of $G$: 
\begin{equation} \label{eq:tau vs secondsum}
\tau(G) = \eta(G) + \secondsum(G) \, 
\end{equation}
(see \cite[Proposition 2.9]{cinkir-graph}, \cite[Theorem 11.2]{dJSh2}).

\subsection{Proof of Theorem~\ref{thm:main-3}} 
\label{sec:thmC}
In this section we prove Theorem~\ref{thm:main-3} from the introduction.
We start with an identity relating the counts of two-forests and spanning trees to resistances {\em and} $j$-functions.

\begin{prop}
\label{prop:forest-resistance}
Let $G = (V, E)$ be a connected graph. For any fixed $q\in V$ we have
\begin{equation}
	\frac{\kappa_2(G)}{\kappa(G)} = \sum_{v\in V} r(v,q) + \sum_{e\in E} j_q(e^+,e^-)^2 - \sum_{e\in E} r(e^+,q) r(e^-,q) \, .
\end{equation}
\end{prop}
\begin{proof}
By Proposition~\ref{prop:forest-count} we have
\begin{equation}
\label{eq:14-1}
	\frac{\kappa_2(G)}{\kappa(G)} = \sum_{v \in V} \frac{\kappa_2(v|q)}{\kappa(G)} - \sum_{e \in E} \frac{\kappa_3(e^+| e^-| q)}{\kappa(G)} \, ,
\end{equation}
and by Proposition~\ref{prop:3-forest} we have
\begin{equation} 
\label{3-forest-divided}
	\frac{\kappa_3(e^+| e^-| q)}{\kappa(G)} 
	= \frac{\kappa_2(e^+|q)}{\kappa(G)} \cdot \frac{\kappa_2(e^-|q)}{\kappa(G)} - \left( \frac{\kappa_2(e^+e^-|q)}{\kappa(G)}\right)^2 \, .
\end{equation}
The result follows from substituting \eqref{3-forest-divided} in \eqref{eq:14-1},
and applying \eqref{eq:resistance-trees} and \eqref{eq:potential-trees}.
\end{proof}

We can now prove the main technical identity in this paper, which is equivalent to Theorem~\ref{thm:main-3} in the introduction.
\begin{thm}\label{thm:tau-to-forests}
Let $G$ be a connected graph. Then
\[
	\frac{\kappa_2(G)}{\kappa(G)} = 3 \secondsum(G) + \frac1{4}
\sum_{e\in E} r(e^+,e^-)^2 \,  .
\]
\end{thm}
\begin{proof}
It follows from the defining equation \eqref{eq:gamma-def} and Proposition~\ref{prop:res-identities} (e) that
 
\begin{align}\label{eq:22-2-1}
  \secondsum(G) &=  \frac14 \sum_{e\in E} r(e^+,e^-)^2 + \sum_{e\in E} r(e^+,e^-) j_q(e^+,e^-) \\
  &\quad\quad + \sum_{e\in E} j_q(e^+,e^-)^2 - \sum_{e\in E} r(e^+, q) r(e^-, q) \, .
\end{align}

By Proposition~\ref{prop:tau-equiv} (c), we also know $2\secondsum(G)$ equals the $q$-th entry of $R\boldmu$. Therefore
\begin{align} \label{eq:aux}
	2\secondsum(G) &= \sum_{v\in V} r(v,q) \left(1 - \frac12 \sum_{e \in \Nedge(v)} r(e^+,e^-)\right) \\
	&= \sum_{v\in V} r(v,q) - \frac{1}{2} \sum_{v\in V} r(v,q) \sum_{e \in \Nedge(v)}  r(e^+,e^-) \\
	&= \sum_{v\in V} r(v,q) - \frac12 \sum_{e\in E} r(e^+,e^-) \left(r(e^+,q) + r(e^-,q)\right) \, .
\end{align}

Using Proposition~\ref{prop:res-identities} (b), the last line of \eqref{eq:aux} can be rewritten as
\begin{equation}\label{eq:22-1-1}
2\secondsum(G) = \sum_{v\in V} r(v,q) - \frac12 \sum_{e\in E} r(e^+,e^-)^2 - \sum_{e\in E} r(e^+,e^-) j_q(e^+,e^-)  \, . 
\end{equation}

After adding the two identities \eqref{eq:22-2-1} and \eqref{eq:22-1-1} and cancellations, we obtain 
\begin{equation}
  3\secondsum(G) = 
   \sum_{v\in V} r(v,q) - \frac14 \sum_{e\in E} r(e^+,e^-)^2 + \sum_{e\in E} j_q(e^+,e^-)^2 - \sum_{e\in E} r(e^+,q)r(e^-,q) \, ,
\end{equation}
which we rearrange as
\begin{multline}
\label{eq:23-1}
3\secondsum(G) + \frac14 \sum_{e\in E} r(e^+,e^-)^2 
 = \sum_{v\in V} r(v,q) + \sum_{e\in E} j_q(e^+,e^-)^2 - \sum_{e\in E} r(e^+,q)r(e^-,q)\, .
\end{multline}
Finally, applying Proposition~\ref{prop:forest-resistance} to the right-hand side yields
\[
	3\secondsum(G) + \frac14 \sum_{e\in E} r(e^+,e^-)^2  = \frac{\kappa_2(G)}{\kappa(G)} \, . \qedhere
\]
\end{proof}

\begin{eg}
Suppose $G$ is the house graph (Figure~\ref{fig:house}). 
Then ${\kappa_2(G)}/{\kappa(G)} = {19}/{11}$, while
\[ 
	\sum_{e\in E} r(e^+,e^-)^2 = {\textstyle \frac{(8^2 + 8^2 + 8^2 + 6^2 + 7^2 + 7^2)}{11^2} = \frac{326}{121}
	\qquad\text{and}\qquad 
	\secondsum(G) = \frac12 \boldmu^\tr R \boldmu  = \frac{85}{242}}. 
\] 
One can readily verify the identity in Theorem~\ref{thm:tau-to-forests}. 

Furthermore, $\tau(G) = 9/22$ (this can be computed directly using \eqref{eq:tau vs secondsum}), $g=2$, and $|E| = 6 $. One can then verify the equivalent identity \eqref{eq:tau} given below.
\end{eg}

\begin{rmk}\label{rmk:tau2}
In terms of the tau constant, we have the following identity:
\begin{equation}\label{eq:tau}
\tau(G) = \frac13 \frac{\kappa_2(G)}{\kappa(G)} - \frac{1}{12}|E| + \frac{1}{6} g \, ,
\end{equation} 
where $g = |E|-|V|+1$ denotes the genus of $G$. 
This follows from Theorem~\ref{thm:tau-to-forests}, equation \eqref{eq:tau vs secondsum}, and Foster's identity (Theorem~\ref{thm:foster}).
\end{rmk}

\subsection{Graphs with edge lengths}

We mention that Theorem~\ref{thm:main-3} has a straightforward extension to graphs with {\em edge lengths}. 
Let $(G, \ell)$ be a graph together with a function $\ell : E \to \RR_{>0}$, assigning a positive real length to each edge.
Then is natural to define the {\em weighted counts} of spanning trees and two-forests as follows:
\[
	\kappa(G; \ell) = \sum_{T \in \mathcal F_1(G)} \prod_{e \not\in T} \ell(e) , 
	\qquad\qquad
	\kappa_2(G; \ell) = \sum_{F \in \mathcal F_2(G)} \prod_{e \not\in F} \ell(e) \,.
\]
The {\em resistance function} of a graph $(G, \ell)$ is defined by taking the effective resistance $r(u, v)$ that results from attaching a resistor of resistance $\ell(e)$ on each edge~$e$.
Then
\begin{equation}
\label{eq:main-3-lengths}
	\frac{\kappa_2(G; \ell)}{\kappa(G; \ell)} 
	= \frac{3}{4} \sum_{e \in E} \frac{\left( r(e^+,q) - r(e^-,q) \right)^2} {\ell(e)}
	+ \frac1{4}
\sum_{e\in E} \frac{r(e^+,e^-)^2}{\ell(e)} \,  .
\end{equation}
We omit the details of verifying \eqref{eq:main-3-lengths}, which are straightforward to adapt from our arguments above, using the length-adjusted notion of curvature 
\[
\boldmu_x = 1 - \frac12 \sum_{e \in \Nedge(x)} \frac{r(e^+, e^-)}{\ell(e)}\, .
\]

\begin{rmk}
Let $\Gamma$ be the metric graph obtained from $(G, \ell)$ by replacing each edge $e$ of $G$ with a line segment of length $\ell(e)$. 
In~\cite[equation (46)]{cinkir-inventiones}, Cinkir defines the constant $y(\Gamma)$ to be the expression on the right-hand side of \eqref{eq:main-3-lengths}
and proves various properties satisfied by $y(\Gamma)$.
However, this previous work does not make the connection to two-forests.
\end{rmk}

\section{Proofs of Theorems~\ref{thm:main-1} and \ref{thm:main-2}} 
\label{sec:thmAB}
In this section we prove the main theorems given in the introduction.
\begin{thm}
\label{thm:main-two-forests}
For $G = (V,E)$
a connected graph,
\[
	\frac{\kappa_2(G)}{\kappa(G)} \geq \frac{(|V|-1)^2}{4|E|} \, .
\]
\end{thm}
\begin{proof}
By Theorem~\ref{thm:tau-to-forests}, and the fact that $\secondsum(G) \geq 0$, we obtain
\[ 
	\frac{\kappa_2(G)}{\kappa(G)} 
	\;=\; 3\secondsum(G) + \frac1{4} \sum_{E} r(e^+,e^-)^2  
	\;\geq\; \frac1{4} \sum_{E} r(e^+,e^-)^2 \, .
\]
The result now follows from Corollary~\ref{cor:res-squared}.
\end{proof}

Given a graph $G=(V,E)$, let $\avgdeg$ denote its average degree:
\[
	\avgdeg =  \frac{1}{|V|}\sum_{v \in V} \deg{v} \, .
\]
By the ``handshaking lemma'', we have $\avgdeg = 2|E|/|V|$. 

\begin{thm}
\label{thm:main-cut-set}
Let $G = (V, E)$ be a connected graph.
For a uniformly random two-forest $F$, the expected size of the cut set $\cut F$ satisfies
\begin{equation}
	\EE(|\cut F|) \leq 2(\avgdeg) \left(1 + \frac1{|V|-1}\right) \, .
\end{equation}
\end{thm}

\begin{proof}
Lemma~\ref{lem:average-cut} states that
\[ 
	\mathbb E(|\cut F|) = \frac{\kappa(G)(|V|-1)}{\kappa_2(G)} \, ,
\]
and Theorem~\ref{thm:main-two-forests} implies the bound
\begin{equation}
\frac{\kappa(G) (|V|-1)}{\kappa_2(G)} \leq \frac{4|E|}{|V|-1} \, .
\end{equation}
It follows that
\[ 
  \mathbb E(|\cut F|) \leq 2\left(\frac{2|E|}{|V|}\right) \left( \frac{|V|}{|V|-1}\right) 
  = 2 (\avgdeg) \left( 1 + \frac1{|V|-1}\right) \, . \qedhere
\]
\end{proof}

\section{Examples}\label{sec:examples}

\begin{eg}[Complete graph]
Consider $G = K_n$, the complete graph on $n$ vertices.
The number of spanning trees, due to Cayley, is
\begin{equation}
\kappa(K_n) = n^{n-2} \, ,
\end{equation}
while the number of two-forests is 
\begin{align}
\kappa_2(K_n) 
&= \frac12 (n-1)(n+6) n^{n-4} \, ,
\end{align}
see Liu--Chow~\cite[Equation (22)]{liu-chow}.
Thus for this graph, the two-forest ratio is
\begin{equation}
\frac{\kappa_2(K_n)}{\kappa(K_n)} = \frac12 \left(1 - \frac1{n}\right) \left(1 + \frac{6}{n}\right) \, .
\end{equation}
In comparison, the lower bound for $\kappa_2(K_n) / \kappa(K_n)$ given by Theorem~\ref{thm:main-two-forests} is
\begin{equation}
\frac{(|V|-1)^2}{4|E|} = \frac{(n-1)^2}{2n(n-1)}
= \frac12 \left(1 - \frac1{n}\right) \, .
\end{equation}
Hence our bound gives the correct leading term as $n \to \infty$.

Regarding the average edge cut size of a random two-forest,
for the complete graph we have
\begin{equation}
\EE(|\cut F|) = \frac{\kappa(K_n)(|V|-1)}{\kappa_2(K_n)}
= 2n\left(1 - \frac{6}{n+6}\right).
\end{equation}
Since $K_n$ is $(n-1)$-regular, the inequality \eqref{eq:main-2-regular} gives the upper bound for $\EE(|\cut F|)$
\begin{equation}
2d \left(1 + \frac1{n-1}\right) = 2n \, .
\end{equation}
\end{eg}

\begin{eg}[Cycle graph]
Let $C_n$ denote the cycle graph, which has $n$ vertices and $n$ edges.
The number of spanning trees and two-forests of $C_n$ are given by
$
	\kappa(C_n) = n 
$ and $
	\kappa_2(C_n) = \binom{n}{2} ,
$
so
\[\displaystyle
	\frac{\kappa_2(C_n)}{\kappa(C_n)} = \frac12 n - \frac12 \, .
\]
For the cycle graphs, our results give the bound 
\[
\frac{\kappa_2(C_n)}{\kappa(C_n)} \geq \frac14 n + O(1) \qquad\text{as }n \to \infty \, .
\]
Here the lower bound does not match the leading term on the left as $n \to \infty$.
Theorem~\ref{thm:main-two-forests} ignores the $\secondsum(C_n)$ term of Theorem~\ref{thm:tau-to-forests}, which is not negligible asymptotically in this case.
\end{eg}

\begin{eg}[Wheel graph]
Let $W_n$ denote the wheel graph, which consists of an $n$-cycle and an additional ``central'' vertex connected to all cycle vertices.
The graph $W_n$ has $n + 1$ vertices and $2n$ edges.
\begin{figure}[h]
\centering
	\begin{tikzpicture}[scale=0.6]
	\coordinate (0) at (0,0);
	\coordinate (4) at (90:2);
	\coordinate (5) at (162:2);
	\coordinate (1) at (234:2);
	\coordinate (2) at (-54:2);
	\coordinate (3) at (18:2);
	
	\foreach \c in {0,1,2,3,4,5} {
		\filldraw[black] (\c) circle (2pt);
	}

	\draw (1) -- (2) -- (3) -- (4) -- (5) -- (1) -- (0);
	\draw (2) -- (0) -- (3);
	\draw (4) -- (0) -- (5);
	\end{tikzpicture}
\caption{Wheel graph $W_5$.}
\end{figure}
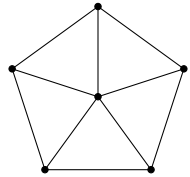
\begin{table}[h]
\[
	\begin{array}{c|ccccccccc}
	n & 1 & 2 & 3 & 4 & 5 & 6 & 7 & 8 \\ \hline \\[-0.8em]
	\kappa(W_n) & 1 & 5 & 16 & 45 & 121 & 320 & 841 & 2205 \\
	\kappa_2(W_n) & 1 & 4 & 15 & 52 & 170 & 534 & 1631 & 4880
\end{array}
\]
\caption{Number of spanning trees and two-forests for the wheel graphs.}
\label{tab:wheel}
\end{table}
The values $\kappa(W_n)$ and $\kappa_2(W_n)$ are listed in Table~\ref{tab:wheel} for small $n$ 
(see OEIS sequences A004146 and A117202 \cite{oeis}).

We have
\[
	\kappa(W_n) = \left( \frac{1 + \sqrt{5}}{2} \right)^{2n} + \left( \frac{1 - \sqrt{5}}{2} \right)^{2n} - 2 \, 
\]
(see Myers~\cite{myers}).
The number of two-forests is
\[
	\kappa_2(W_n) = \frac{n}{\sqrt{5}} \left( \frac{1 + \sqrt{5}}{2} \right)^{2n - 1} - \frac{n}{\sqrt{5}} \left( \frac{1 - \sqrt{5}}{2} \right)^{2n - 1} \, .
\]
Thus as $n \to \infty$,
\[
	\frac{\kappa_2(W_n)}{\kappa(W_n)} = \frac{1}{C} n + o(1) 
	\qquad\text{where } C = \sqrt{5} \cdot \left(\frac{1 + \sqrt{5}}{2}\right) \approx 3.618 \, .
\]
In comparison, Theorem~\ref{thm:main-two-forests} gives the lower bound
\[
	\frac{\kappa_2(W_n)}{\kappa(W_n)} \geq \frac{(|V| - 1)^2}{4|E|} = \frac{1}{8} n \, .
\]
\end{eg}

\begin{eg}[Toroidal grid graph]
Let $T_{n, n}$ denote the $n\times n$ toroidal grid graph, which is the nearest-neighbor graph on $\ZZ^2$ modulo $(n\ZZ)^2$.

For $G = T_{n, n}$, we have
$|V| = n^2$ and $|E| = 2n^2$. Our theorems show that 
\[
	\frac{\kappa_2(G)}{\kappa(G)} \geq 
	\frac{(n^2-1)^2}{ 8n^2}
	= \frac{n^2}{8}\left( 1 - \frac1{n^2}\right)^2
	\qquad\text{and}\qquad
	\EE(|\cut F|)\leq 8\left(1 + \frac{1}{n^2-1}\right) \, .
\]
Kassel, Kenyon, and Wu~\cite{kassel-kenyon-wu} show that for the $n\times n$ toroidal grid graph, $\EE(|\cut F|) = 8 + o(1)$ as $n\to \infty$
so our bound in sharp in this family of graphs.
\end{eg}

In \cite{kassel-kenyon-wu} one also finds the $r$-dimensional generalization (for any $r \geq 2$):
the nearest-neighbor graph in $\ZZ^r$ modulo $(n \ZZ)^r$ satisfies $\EE(|\cut F|) = 4r + o(1)$ as $n \to \infty$. 
Our bound is also sharp in this family.

\vspace{1cm}
\end{document}